\documentclass[11pt]{amsart}
\usepackage{amssymb, amsmath, amsthm, mathrsfs, amsopn}
\usepackage{graphicx}
\usepackage[a4paper, centering]{geometry}
\newcommand{\A}{\boldsymbol{A}}
\newcommand{\mean}[1]{\,-\hskip-1.08em\int_{#1}} 
\newcommand{\Leb}[1]{\mathcal{L}^{#1}} 
\newcommand{\B}{\boldsymbol B}
\newcommand{\bder}{\boldsymbol b}

\def\bbbr{\mathbb{R}}
\def\R{\mathbb{R}}
\newcommand{\eeta}{\boldsymbol \eta}

\def\u{{u}}
\def\uu{{\bf u}}
\def\z{{\bf z}}
\def\DM{{\mathcal{DM}^{\infty}}}

\def\v{{\bf v}}
\def\uh{\hat{u}}
\def\HH{{\bf{A}}}
\def\w{{t}}

\def\Xi{\chi^*_{[a_i, a_{i+1}]}}

\DeclareMathOperator{\diver}{div}
\DeclareMathOperator{\Div}{div}

\newcommand{\defeq} {:=}
\newcommand{\medint}{-\kern  -,375cm\int}
\newcommand{\medintinrigo}{-\kern  -,315cm\int}

 \newcommand{\hh}{{\mathcal H}^{N-1}}

\newcommand{\LLN}{{\mathcal L}^N}
\newcommand{\LLU}{{\mathcal L}^1}

\newcommand{\M}[1]{\mathcal{#1}}    
\renewcommand{\H}{\M{H}}
\newcommand{\Haus}[1]{{\mathcal H}^{#1}} 

\newcommand{\res}{\mathop{\hbox{\vrule height 7pt width .5pt depth 0pt
\vrule height .5pt width 6pt depth 0pt}}\nolimits} 

\def\pscal#1#2{\left\langle #1\,,\, #2 \right\rangle}
\DeclareMathOperator{\sign}{sign}
\def\ut{\widetilde{u}}

\newtheorem{definition}{Definition}[section]

\newtheorem{lemma}[definition]{Lemma}

\newtheorem{theorem}[definition]{Theorem}

\newtheorem{proposition}[definition]{Proposition}

\newtheorem{corollary}[definition]{Corollary}
\theoremstyle{remark}
\newtheorem{remark}[definition]{Remark}

\makeatletter
\def\@settitle{\begin{center}%
		\baselineskip14\p@\relax
		\bfseries
		\uppercasenonmath\@title
		\@title
		\ifx\@subtitle\@empty\else
		\\[5ex]
		\normalsize\mdseries\@subtitle
		\fi
	\end{center}%
}
\def\subtitle#1{\gdef\@subtitle{#1}}
\def\@subtitle{}
\makeatother

\begin{document}
\title[On the chain rule formula for divergences]
{On the chain rule formulas for divergences\\ and applications to conservation laws}

\subtitle{On the occasion of the 60th birthday of Nicola Fusco}

\author[G.~Crasta]{Graziano Crasta}
\address{Dipartimento di Matematica ``G.\ Castelnuovo'', Univ.\ di Roma I\\
	P.le A.\ Moro 2 -- I-00185 Roma (Italy)}
\email{crasta@mat.uniroma1.it}
\author[V.~De Cicco]{Virginia De Cicco}
\address{Dipartimento di Scienze di Base  e Applicate per l'Ingegneria, Univ.\ di Roma I\\
	Via A.\ Scarpa 10 -- I-00185 Roma (Italy)}
\email{virginia.decicco@sbai.uniroma1.it}

\date{\today}

\begin{abstract}
In this paper we prove a nonautonomous chain rule formula for the
distributional divergence of the
composite function $\v(x)=\B(x,u(x))$,
where $\B(\cdot,t)$ is a divergence--measure vector field
and $u$ is a function of bounded variation.
As an application, we prove a uniqueness result for scalar conservation laws
with discontinuous flux.
\end{abstract}

\maketitle

\section{Introduction}

Nonautonomous chain rules formulas in $BV$ 
have been successfully used in the study of
semicontinuity properties of integral functionals
(see \cite{DC,dcfv,DCFV2,dcl}) and
conservation laws with discontinuous flux of the form
\begin{equation}\label{f:sc1}
u_t + \Div \B(x, u) = 0,
\qquad (t,x)\in (0,+\infty)\times\R^N
\end{equation}
(see \cite{CD,CDD,CDDG}
and also \cite{vol,vol1} in the autonomous case).
In this paper we shall restrict our attention only to this second kind of application.

In order to clarify the connection between chain rule formulas and
uniqueness results for the Cauchy problems
associated with \eqref{f:sc1},
it will be convenient to recall some previous results.

In \cite{CDD} the authors considered a flux $\B$ such that
$\B(\cdot,z)$ is a special function of bounded variation (SBV)
and of class $C^1$ with respect to the second variable.
A uniqueness result for~\eqref{f:sc1} is then obtained in the class of $BV$ functions
by using the chain rule formula proven in \cite{ACDD} for the
composite function $\v(x) := \B(x, u(x))$.
(For the sake of completeness we recall that,
under the same structural hypotheses on the flux, 
a similar uniqueness result has been recently obtained in \cite{CDDG}  
for weak entropy solutions, without the $BV$ regularity requirement.)

On the other hand, Panov proved in \cite{Pan1} 
an existence result of entropy solutions 
in the case of discontinuous fluxes $\B(x,z)$ such that $\B(\cdot,z)$
is a vector field whose distributional divergence $\Div_x \B(\cdot, z)$
is a measure (see \cite{Anz,cf,ChenTorres} for a general theory of bounded divergence-measure vector field).
This assumption on $\Div_x \B(\cdot, z)$, 
rather than requiring $\B(\cdot, z) \in SBV$,
is indeed natural when looking for entropy solutions of \eqref{f:sc1}.

The structure of the proof of the uniqueness result in \cite{CDD}
can be adapted to this more general situation, 
provided that one can prove a 
suitable chain rule formula.
This is exactly the aim of this paper:
in Section~\ref{s:div} we shall prove a 
nonautonomous chain rule formula for 
the divergence of the vector field
$\v(x) := \B(x, u(x))$,
where $\B(\cdot, t)$ is a
divergence--measure vector field, 
of class $C^1$ with respect to the second variable,
and $u\colon\R^N\to\R$ is a function of bounded variation.
Then, we can mimic the proof in \cite{CDD} in order to obtain,
under these assumptions on $\B$,
a uniqueness result for $BV$ solutions
of the Cauchy problems associated with \eqref{f:sc1}
(see Section~\ref{s:cons}).
We stress that this is not a genuine well-posedness
result, since uniqueness of solutions has been proven
in a class of functions 
which is smaller
than the one for which existence
has been obtained by Panov.

\medskip
Before stating our results in a more precise way,
let us recall the state of the art
about chain rule formulas,
starting from the autonomous (i.e., independent of $x$) case.

The first result concerning distributional derivatives
is the one proved by
Vol'pert in \cite{vol}  (see also \cite{vol1}), in view of applications to the study of quasilinear hyperbolic equations.
He established a chain rule formula for distributional derivatives of the composite function $v(x)=B(u(x))$, where $u:\Omega\to\R$ has bounded variation in the open subset $\Omega$ of $\R^{N}$ and $B:\R\to\R$ is continuously differentiable. 
He proved that $v$ has bounded variation and its distributional derivative 
$Dv$ (which is a Radon measure on $\Omega$) admits an explicit representation in terms of the classical derivative $B'$ and of the distributional derivative $Du$\,.
More precisely, the following equality holds  
\begin{equation}\label{chainDv}
Dv= B' (u)\nabla u \ \mathcal L^{N}+ B'(\widetilde u) D^cu+[B(u^+)-B(u^-)]\, \nu_u\, \mathcal H^{N-1}
\res{J_{u}}\,,
\end{equation}
in the sense of measures,
where
\[
Du=\nabla u \ \mathcal L^{N}+D^cu+
(u^+ - u^-) 
\nu_u\, \mathcal H^{N-1}
\res{J_{u}}\,
\]
is the decomposition of $Du$ into its absolutely continuous part 
$\nabla u\, \LLN$ with respect to the Lebesgue measure $\mathcal L^{N}$, its Cantor part $D^cu$ and its jump part, which in turn is a measure concentrated on
the $\mathcal{H}^{N-1}$--rectifiable jump set $J_u$ of $u$. 
Here, $\nu_u$ denotes the measure theoretical unit normal to $J_u$, $\widetilde u$ is the approximate limit 
of $u$ and $u^+$, $u^-$ are the traces of $u$ on $J_u$.
(Here and in the following we refer to Chapter~3 of \cite{AFP} for notations and the basic facts
concerning $BV$ functions.)

An identity similar to \eqref{chainDv} holds also in the vectorial case (see Theorem 3.96 in \cite{AFP}), namely when $\uu:\R^{N}\to\R^{h}$ has bounded variation  and $B:\R^{h}\to\R$ is continuously differentiable. 
In this case, \eqref{chainDv} can be written as
\begin{equation}\label{chainDv2}
D v=\nabla B(\uu)\nabla \uu \ \mathcal L^{N} + \nabla B(\widetilde \uu) D^c\uu+
[B(\uu^+)-B(\uu^-)] \nu_\uu\, \mathcal H^{N-1}
\res{J_{\uu}}\,.
\end{equation}

A further extension, that we are not going to use in the present paper,
concerns the case when $B$ is only a Lipschitz continuous function.
In this case, a general form of the formula was proved by Ambrosio and Dal Maso in  \cite{adm} 
(see also \cite{LM}, Theorem 3.99 in \cite{AFP} for the scalar case
and \cite{DCinpreparation} for the nonautonomous case).

\medskip
Recently, analogous chain rule formulas have been obtained
in the case of an explicit dependence with respect to the space variables $x$, especially in view of applications to semicontinuity results for convex integral nonautonomous
functionals (see \cite{DC,dcfv,DCFV2,dcl}) and to conservation laws with discontinuous flux (see \cite{CD,CDD,CDDG}). 
This amounts to describe the
distributional derivative of the composite function $\v(x)=\B(x,\uu(x))$, 
where $\B(x,\cdot)$ is continuously differentiable and
$\B(\cdot,\z)$ and $\uu$ are functions with low regularity 
(which will be specified later).
These formulas contain another derivation term due to the presence of the explicit dependence on $x$. 

In the case of $\uu$ and $\B$ regular functions, the classical chain rule formula 
\[
\nabla \v(x)=\nabla_{x}\B (x,\uu(x))+ \nabla_{\z}\B(x,\uu(x))\cdot\nabla \uu(x)\,,\quad x\in\R^{N}\,,
\]
is a pointwise identity and the derivatives here occurring are the classical ones.

Clearly, when $\B(\cdot, \z)$ or $\uu$ (or both) are not regular functions,
then also $\v$ need not be regular and a number of extra terms will appear,
as in the previous formulas \eqref{chainDv} and \eqref{chainDv2}.

In the main result of the paper,
stated in Theorem~\ref{chainb3}, we assume that
$h=1$, the function $\B(x,\cdot)$ is $C^1$, 
$\Div_x\B(\cdot,z)$ 
is a Radon measure,
and $u$ is a scalar function of bounded variation. 
As we shall see, in order to obtain the formula in such generality,
we need to assume \textsl{a-priori} the existence of the strong traces $\B^\pm(\cdot, z)$ of $\B(\cdot, z)$
on a $\mathcal{H}^{N-1}$--rectifiable universal singular set 
$\mathcal{N}$, independent of $z$ (see Section~\ref{s:div}).
Nevertheless, we think that this restriction should not be
too much severe in view of applications to conservation laws with
discontinuous flux (see Section~\ref{s:cons} and Remark~\ref{r:triang}).
Clearly, in this case one can expect a chain
rule formula only for the divergence of the
composite function \(\v(x) = \B(x, u(x)))\).
Namely,
we shall prove that
the distributional divergence of $\v$
is given by
\begin{equation}\label{lolo5}
\begin{split}
\Div \v(x))=&\left.\Div_x\B(x,t)\right|_{t=u(x)}\\
+&\pscal{\bder(x, \ut(x))}{\widetilde Du}
+\pscal{
	\B^*(x,u^+)
	-\B^*(x,u^-)
}{\nu_u}
\hh\res {\mathcal N\cup J_u}
\end{split}
\end{equation}
in the sense of measure, where $\B^*(x,z)=\frac12[\B^+(x,z)+\B^-(x,z)]$ and the measure $\Div_x\B(x,z)$, depending on the parameter $z$, is computed in ${z=u(x)}$ in a suitable sense (see Remark \ref{remrem}).
The proof is based on the regularization argument used in \cite{DCFV2}. 

We shall see that, when $u$ is a function of bounded variation,
this explicit formula for $\Div \v$ can be used to obtain uniqueness results 
for \eqref{f:sc1}.
This improves the analogous uniqueness result previously obtained in \cite{CDD},
using the chain rule formula for $\v$,  
assuming $\B(\cdot, z)$ a special function of
bounded variation. 
We emphasize that, interestingly enough,
the proof of the uniqueness result of \cite{CDD}
can be retraced with minor modifications in
this new setting, since what is really needed
is only the chain rule formula for
$\Div \v$.

\medskip
The structure of the paper is the following.
First of all, in Sections~\ref{s:ureg} and~\ref{s:BV} we review
some known chain rule formulas that we believe will help the
reader understanding the more general formula of Section~\ref{s:div}.

More precisely, in Section~\ref{s:ureg} we recall the results proved
in the case $\B(\cdot, \z)$ not regular,
but $\uu$ regular enough (i.e., in the Sobolev space $W^{1,1}$).
These results have been proved in \cite{dcl} and \cite{DCFV2}, assuming   
that $\Div_x \B(\cdot, \z)$ belongs to $L^1$ or 
to the space of Radon measures respectively.

In Section~\ref{s:BV} we review the chain rule formulas in the case
when both $\B(\cdot, \z)$ and $\uu$ are of bounded variation, 
recalling the results proved in \cite{DCFV2} and \cite{ACDD}
in the case of scalar and vector functions respectively.

In Section~\ref{s:div} we prove the main result of the paper,
stated in Theorem~\ref{chainb3}, 
concerning the distributional divergence of the composite function $\v(x)=\B(x,u(x))$.

Finally, in Section~\ref{s:cons}
we consider some applications of Theorem~\ref{chainb3}
to the uniqueness issue for the Cauchy problems
for multidimensional scalar conservation laws with discontinuous flux.

\section{Nonautonomous chain rules for $u\in W^{1,1}$}
\label{s:ureg}

As we have already said in the Introduction,
this section and the next one are devoted to the review
of some known chain rule formulas.

In the following, $\Omega$ will always denote a nonempty open subset
of $\R^N$.

\subsection{Vectorial case $\mathbf{u}\in W^{1,1}(\Omega;\mathbb{R}^{h})$}
The first formula of this type is established in \cite{dcl} for functions $u\in W^{1,1}(\R^{N};\R^{h})$ by assuming that, for every $\z\in \R^{h}$, $\B(\cdot,\z)$ is an $L^1$ function whose distributional divergence belongs to $L^1$. 
(In particular, this condition holds if $\B(\cdot,\z)\in W^{1,1}(\R^{N};\R^{h})$.)

Let us consider the space
\[
L^{1}({\operatorname*{div};}\Omega)=\left\{  \mathbf{u}\in L^{1}(\Omega
;\mathbb{R}^{N}):\operatorname*{div}\mathbf{u}\in L^{1}(\Omega)\right\}  ,
\]
where $\operatorname*{div}\mathbf{u}$ denotes the distributional divergence of \(\uu\).

We recall that an $\mathcal{H}^{N-1}$--measurable set $A\subset\R^N$ is said to be 
countably $\mathcal{H}^{N-1}$--rectifiable 
if it can be covered, up to a set of vanishing $\mathcal{H}^{N-1}$ measure,
by a sequence of $C^1$ hypersurfaces.
The set $A$ is said to be purely $(N-1)$-unrectifiable
if $\Haus{N-1}(A\cap\Gamma)=0$ whenever $\Gamma$ is countably
$\Haus{N-1}$-rectifiable.

\begin{theorem}
[$L^{1}({\operatorname*{div};}\Omega)$-dependence]\label{weak theorem}
Let $\B:\Omega\times\mathbb{R}^{h}\rightarrow\mathbb{R}^{N}$ be a Borel function. 
Assume
that there exist an $\mathcal{L}^{N}$-null set $\mathcal{N}_0\subset\Omega$ and
a purely $\mathcal{H}^{1}$-unrectifiable set $\mathcal{M}\subset\mathbb{R}%
^{h}$ such that

\begin{enumerate}
\item[(i)] for all $\mathbf{z}\in\R^{h}$ the function 
$\B\left(
\cdot,\mathbf{z}\right)  \in L_{\operatorname*{loc}}^{1}({\operatorname*{div}
;}\Omega)$;

\item[(ii)] for all $x\in\Omega\setminus\mathcal{N}_0$ the function
$\operatorname*{div}\nolimits_{x}\B(x,\cdot)$ is approximately continuous in
$\mathbb{R}^{h}$;

\item[(iii)] for all $x\in\Omega\setminus\mathcal{N}_0$ the function $\B\left(
x,\cdot\right)  $ is differentiable in $\R^{h}\setminus\mathcal{M}$
and approximately continuous in $\mathcal{M}$;

\item[(iv)] for every $\Omega^{\prime}\times D\subset\subset\Omega
\times\mathbb{R}^{h}$ there exist $g\in L^{1}(\Omega)$ and $L>0$ such
that
\[
\left|  \B(x,\mathbf{z})\right|  +\left|  \operatorname*{div}\nolimits_{x}%
\B(x,\mathbf{z})\right|  \leq g(x)
\]
for $\mathcal{L}^{N}$--a.e.\ $x\in\Omega^{\prime}$ and for all
$\mathbf{z}\in D,$ and
\[
\left|  \nabla_{\mathbf{z}}\B(x,\mathbf{z})\right|  \leq L
\]
for $\mathcal{L}^{N}$--a.e.\ $x\in\Omega^{\prime}$ and for all
$\mathbf{z}\in D\setminus\mathcal{M}$.
\end{enumerate}

Then for every $\uu\in W^{1,1}(\Omega;\mathbb{R}^{h})\cap
L_{\operatorname*{loc}}^{\infty}(\Omega;\mathbb{R}^{h})$ the function
$\mathbf{v}:\Omega\rightarrow\mathbb{R}^{N}$, defined by
\[
\mathbf{v}(x)=\B(x,\uu(x))\quad x\in\Omega,
\]
belongs to $L_{\operatorname*{loc}}^{1}({\operatorname*{div};}\Omega)$ and
\[
\operatorname*{div}\mathbf{v}(x)=\operatorname*{div}\nolimits_{x}%
\B(x,\uu(x))+\operatorname*{tr}\left(  \nabla_{\z}%
\B(x,\uu(x))\,\nabla\uu(x)\right)
\]
for $\mathcal{L}^{N}$--a.e.\ $x\in\Omega$, provided 
$\nabla_{\z}\B(x,\mathbf{u}(x))\,\nabla\mathbf{u}(x)$ is interpreted to be zero whenever
$\nabla\mathbf{u}(x)=0$, irrespective of whether $\nabla_{\z}\B(x,\mathbf{u}(x))$ is defined.
\end{theorem}

\subsection{Scalar case $u\in W^{1,1}(\Omega)$}
An important special case is given when $\B$ has the form
\begin{equation}
\B(x,u)=\int_{0}^{u}\bder(x,s)\,ds, \label{A}%
\end{equation}
where $\bder\colon\Omega\times\mathbb{R}\rightarrow\mathbb{R}^{N}$.
Clearly, when $\B$ is of class $C^1$ with respect to the second variable,
it is not a real restriction assuming that it is of the form \eqref{A}
for some vector field $\bder$ which is continuous with respect to the same variable.
On the other hand, some assumptions in Theorems~\ref{chainb3reg} 
and~\ref{chainb2} below can be stated in a more polished way in terms of $\bder$.

In this case a formula has been established in \cite{DCFV2} (see Theorem 3.4) for scalar functions 
$u\in W^{1,1}(\Omega)$ by assuming that, for every $t\in \R$, $\B(\cdot,t)$ is an $L^\infty_{\rm{loc}}$ function whose distributional divergence is a Radon measure. 

In the following, we shall denote by $\DM(\Omega)$ the space of all
vector fields belonging to $L^\infty(\Omega;\R^N)$ whose divergence
in the sense of distribution is a Radon measure with finite total variation. 

\begin{theorem}[$\mathcal{DM}^{\infty}$-dependence]\label{chainb3reg}
Let 
$\bder:\Omega\times\bbbr\rightarrow\bbbr^N$ be a locally
bounded Borel function. Assume that
\begin{enumerate}
\item[(i)]  for ${\mathcal L}^N$-a.e.\ $x\in
\Omega$ the function $\bder(x,\cdot)$ is continuous in $\bbbr$;
\item[(ii)] for $\LLU$-a.e.\ $t\in
\bbbr$ the function $\bder(\cdot,t) $ belongs to $\mathcal{DM}^{\infty}(\Omega)$;
\item[(iii)] for any compact set $H\subset\bbbr$,
$$
\int_H|{\rm div}_x\bder(\cdot,t)|(\Omega)\,dt<+\infty\,.
$$
\end{enumerate}
Then, for every $u\in W^{1,1}(\Omega)\cap L^{\infty}(\Omega)$,
the function $\v:\Omega\rightarrow\bbbr^N$,
defined by
\[
\v(x):= \B(x, u(x)) = \int_{0}^{u(x)}\!\bder(x,t)\,dt\,,
\]
belongs to $\mathcal{DM}^{\infty}(\Omega)$ and for any
$\phi\in C_0^1(\Omega)$ we have
\[
\begin{split}
\int_\Omega\pscal{\nabla\phi(x)}{\v(x)}dx=-&\int_{-\infty}^{+\infty}\!\!dt\!\!\int_\Omega{\rm 
sgn}(t)\chi^*_{\Omega_{u,t}}\phi(x)d\diver_{x}\bder(x,t)
\\
-&\int_\Omega\phi(x)\pscal{\bder(x,u(x))}{\nabla 
u(x)}dx
\end{split}
\]
where 
\begin{equation}
\label{f:out}
\Omega_{u,t} :=
\left\{x\in\Omega:\ \text{$t$ belongs to the segment of
endpoints $0$ and $u(x)$}\right\} 
\end{equation}
and for a.e. $t$ the function $\chi^*_{\Omega_{u,t}}$ is the precise representative of the $BV$ function
$\chi_{\Omega_{u,t}}$.
\end{theorem}

\section{Nonautonomous chain rule for $u\in BV$}
\label{s:BV}

\subsection{The scalar case $u\in BV(\R^{N})$}

The case of a scalar function $u\in BV(\R^{N})$  is studied in the papers \cite{dcfv} and \cite{DCFV2}, where it is considered 
again in the case $B = B(x,u)$ of the form \eqref{A}.
In the first paper the authors have established the validity of the chain rule by requiring a $W^{1,1}$ dependence with respect to the variable $x$, 
while in the second one it is assumed only a $BV$ dependence with respect to the variable $x$\,.

\begin{theorem}[$BV$-dependence, see \cite{DCFV2}] 
\label{chainb2}
Let 
$b:\Omega\times\bbbr\rightarrow\bbbr$ be a locally
bounded Borel function. Assume that
\begin{enumerate}
\item[(i)]  for $\LLU$-a.e.\ $t\in
\bbbr$ the function $b(\cdot,t)  \in BV(\Omega)$;
\item[(ii)] for any compact set $H\subset\bbbr$,
\[
\int_H|D_xb(\cdot,t)|(\Omega)\, dt < +\infty\,.
\]
\end{enumerate} 
Then, for every $u\in BV(\Omega)\cap L_{\rm loc}^{\infty}(\Omega)$,
the function $v:\Omega\rightarrow\bbbr$,
defined by
\[
v(x):=\int_{0}^{u(x)}\!b(x,t)\,dt\,,
\]
belongs to $BV_{\rm loc}(\Omega)$ and for any
$\phi\in C_0^1(\Omega)$ we have
\begin{equation}\label{chainrule}
\begin{split}
\int_{\Omega}\!\nabla\phi(x)  v(x)\,dx = {} &
- \int_{-\infty}^{+\infty}\!\!dt\!\!\int_{\Omega}{\rm sgn}(t)
\chi^*_{\Omega_{u,t}}(x)\phi(x)\,dD_xb(x,t)\\
& -
\int_{\Omega}\!\phi(x)b^*(x,\widetilde
u(x))\nabla u(x)\,dx
-\int_{\Omega}\phi(x)b^*(x,\widetilde
u(x))\,dD^cu(x)
 \\
& -
\!\int_{J_u}\!\!\phi(x)\nu_u(x)\,d\mathcal{H}^{N-1}(x)\!\!
\int_{u^-(x)}^{u^+(x)}\!\!b^*(x,t)\,dt,
\end{split}
\end{equation}
where $\Omega_{u,t}$ is the set defined in \eqref{f:out},
$J_u$ is the jump set of $u$, 
and $\chi^*_{\Omega_{u,t}}$ and
$b^*(\cdot,t)$ are, respectively, the precise representatives of
$\chi_{\Omega_{u,t}}$ and $b(\cdot,t)$.
\end{theorem}

Notice that if $b(x,t)\equiv b(t)$, then (\ref{chainrule}) reduces
to the well known chain rule formula for the composition of $BV$
functions with a Lipschitz function, while, in the special case that
$b(x,t)\equiv b(x)$, (\ref{chainrule}) gives the
formula for the derivative of the product of two $BV$ functions.

\subsection{The vectorial case $\uu\in BV_{\rm loc}(\R^{N};\R^h)$}

More recently, a very general formula has been proven in \cite{ACDD} (see also \cite{CD} for $N=1$) for vector functions $\uu\in BV(\R^{N},\R^{h})$.

Here the function $B\colon\R^{N}\times\R^h\to\R$ 
is required to satisfy the following assumptions:
\begin{enumerate}
\item[(a)] $x\mapsto B(x,\z)$ belongs to $BV_{\rm loc}(\R^{N})$ for all
$\z\in\R^h$;
\item[(b)] $\z\mapsto B(x,\z)$ is continuously differentiable in $\R^h$ for almost every $x\in\R^{N}$.
\end{enumerate}
We will use
the notation $C_B$ to denote a Lebesgue negligible set of points
such that $B(x,\cdot)$ is $C^1$ for all $x\in\R^{N}\setminus C_B$.

We assume  that $B$ satisfies, besides (a) and
(b), the following \emph{structural assumptions}:

\begin{enumerate}
\item[(H1)]For some constant $M$, $|\nabla_\z B(x,\z)|\le M$ for all
$x\in\R^{N}\setminus C_B$ and $\z\in\R^h$.
\item[(H2)]For any compact set $H\subset\R^h$ there exists a modulus of continuity
$\tilde\omega_H$ independent of $x$ such that
\[
 |\nabla_\z B(x,\z)-\nabla_\z B(x,\z')|\leq \tilde \omega_H(|\z-\z'|)
 \]
for all $\z,\,\z'\in H$ and $x\in\R^{N}\setminus C_B$.
\item[(H3)]For any compact set $H\subset\R^h$ there exist a positive Radon measure
$\lambda_H$ and a modulus of continuity $\omega_H$ such that
\[
|\widetilde{D}_x B(\cdot,\z)(A)-\widetilde{D}_x B(\cdot,\z')(A)|\le
\omega_H(|\z-\z'|)\lambda_H(A)
\]
for all $\z,\,\z'\in H$ and $A\subset\R^{N}$ Borel.
\item[(H4)] The measure
\begin{equation}\label{defsigma}
\sigma\defeq \bigvee_{\z\in\R^h} |D_x B(\cdot,\z)|,
\end{equation}
(where $\bigvee$ denotes the least upper bound in the space of
nonnegative Borel measures) is finite on compact sets, i.e.\ it is a
Radon measure.
\end{enumerate}

We can now canonically build a countably $\Haus{N-1}$-rectifiable
set $\mathcal N$ containing all jump sets of $B(\cdot,\z)$ as follows.
Indeed, we define
\begin{equation}\label{bsigma}
\mathcal N=\bigl\{x: \limsup_{r\downarrow 0}
\frac{\sigma(B_r(x))}{\omega_{n-1}r^{n-1}}
>0\bigr\}.
\end{equation}
It can be checked that $\mathcal N$ is $\sigma$-finite with
respect to $\Haus{N-1}$ and it is countably $\Haus{N-1}$-rectifiable
(see \cite{ACDD}, Section~2, for details).

\begin{theorem}\label{chain ruleqwqw}
Let $B$ be satisfying (a), (b), (H1)-(H2)-(H3)-(H4) above.
Then 
for any function $\uu\in BV_{\rm loc}(\R^{N};\R^h)$,
the function $\v(x) := B(x,\uu(x))$ belongs to $
BV_{\rm loc}(\R^{N})$ and the following chain rule holds:
\begin{itemize}
\item[(i)] (diffuse part) $|D\v|\ll\sigma+|D\u|$ and, for any Radon measure $\mu$
such that $\sigma+|D\u|\ll\mu$, it holds
\begin{equation}\label{diffuse}
\frac{d\widetilde{D}\v}{d\mu} =\frac{d\widetilde{D}_x B(\cdot,\tilde
\uu(x))}{d\mu} +\nabla_\z\tilde B(x,\tilde \uu(x))\frac{d\widetilde{D}
\uu}{d\mu}\qquad\text{$\mu$-a.e. in $\R^{N}$.}
\end{equation}
\item[(ii)] (jump part)
$J_\v\subset \mathcal N\cup J_\uu$
and, denoting by $u^\pm(x)$ and $B^\pm(x,\z)$ the one-sided traces of
$u$ and $B(\cdot,\z)$ induced by a suitable orientation of $\mathcal
N\cup J_\uu$, it holds
\begin{equation}\label{jump}
D^j \v=\big(B^+(x,\uu^+(x))-B^-(x,\uu^-(x)\big)\nu_{\mathcal N\cup
J_\uu}\Haus{N-1}\res \mathcal (\mathcal N\cup J_\uu)
\end{equation}
in the sense of measures.
\end{itemize}
Moreover for a.e.\ $x$ the map $y\mapsto B(y,\uu(x))$ is
approximately differentiable at $x$ and
\begin{equation}\label{eq:appdif}
\nabla \v(x)=\nabla_xB(x,\uu(x))+\nabla_\mathbf{z}B(x, \uu(x))\nabla \uu(x) \qquad
\text{$\Leb{N}$-a.e.\ in $\R^{N}$\,.}
\end{equation}
\end{theorem}


Here the expression
\[
\frac{d\widetilde{D}_x B(\cdot,\tilde \uu(x))}{d\mu}
\]
means the pointwise density of the measure $\widetilde{D}_x B(\cdot,\z)$
with respect to $\mu$, computed choosing $\z=\tilde \uu(x)$ (notice
that the composition is Borel measurable thanks to the
Scorza-Dragoni Theorem and Lemma 3.9 in \cite{ACDD}). Analogously,
the expression $\tilde B(x,\z)$ is well defined at points $x$ such
that $x\notin S_{B(\cdot,\z)}$ and it can be proved that
$\nabla_\z\tilde B(x,\z)$ is well defined for all $\z$ out of a
countably $\Haus{N-1}$-rectifiable set of points $x$.

%
\section{A generalization to divergence--measure fields}
\label{s:div}

This section is devoted to the proof of the chain rule formula
for divergence--measure vector fields
(see Theorem~\ref{chainb3}).
More precisely, we shall consider the composition
$\v(x) := \B(x, u(x))$ where
$u$ is a function of bounded variation,
$\B(x, \cdot)$ is of class $C^1$
and $\B(\cdot, t)$ is a divergence--measure field.
(See Section~\ref{ss:assumptions} for the complete list of assumptions.)
In order to obtain the formula in this general setting,
we need to assume \textsl{a-priori} the existence of strong traces of $\B$
on a countable
$\mathcal{H}^{N-1}$--rectifiable universal jump set $\mathcal{N}$.
We remark that this assumption is satisfied in the BV setting
recalled in Section~\ref{s:BV}.

Before stating our result,
in Section~\ref{ss:div}
we recall some basic facts on divergence--measure fields.
Then, after listing all the assumptions in Section~\ref{ss:assumptions},
we prove some preliminary results in Section~\ref{ss:prel}.
Finally, in Section~\ref{ss:div2} we prove our main
uniqueness result.

\subsection{Divergence--measure fields }
\label{ss:div}

In what follows,
since the problem is local we shall assume that $\Omega=\R^{N}$.
Moreover, we shall denote by $\DM$ the space of all
vector fields $\A\in L^\infty_{\rm loc}(\R^N;\R^N)$ whose divergence
in the sense of distribution is a Radon measure with locally finite total variation.

For a vector field \(\A\in\DM\)
we shall use the usual decomposition of a measure
\[
\Div \A
=: \Div^a\A\,\LLN+\Div^s\A,
\]
where
$\diver^a\A\,\LLN$ is the absolutely continuous part and
$\diver^s\A$ is the singular part of $\diver\A$ with respect to the Lebesgue measure.
On the other hand, Chen and Frid in \cite{cf} proved that if $\A\in \DM$ 
then $\Div\A\ll\mathcal H^{N-1}$,
hence the singular part can be further decomposed as
\[
\Div^s \A
=: \diver^c\A+\diver^j\A,
\]
where
$\diver^c\A(C) = 0$ if $C$ has $\sigma$--finite
$\mathcal{H}^{N-1}$--measure.
We shall also use the diffuse part $\widetilde{\Div}\A:=\diver^a\A\,\LLN + \diver^c\A$
of the measure $\Div\A$.

In the following, we shall denote by
\(\mathcal{SDM}^{\infty}\) the space of all
vector fields \(\A\in\DM\) such that
\(\Div^c\A = 0\).

Given a domain $\Omega\subset\R^N$ of class $C^1$, 
we can define the 
trace of the normal component of $\A$ on $\partial\Omega$
as a distribution as follows:
\begin{equation}\label{f:ntrace}
\pscal{\text{Tr}(\A, \partial\Omega)}{\varphi}
:= \int_{\Omega} \nabla\varphi \cdot \A\, dx + \int_{\Omega} \varphi\; 
d \Div\A,
\qquad \forall\varphi\in C^\infty_c(\R^N)
\end{equation}
(see \cite{Anz,cf}).

This notion of distributional trace can be extended to any
$\mathcal{H}^{N-1}$--rectifiable set $\mathcal{J} \subset \R^N$.
In particular, it is possible to define the traces
\(\text{Tr}^\pm(\A, \mathcal{J})\) in such a way that
\(\text{Tr}^-(\A, \partial\Omega) = \text{Tr}(\A, \partial\Omega)\)
for every $\Omega \Subset \R^N$ of class $C^1$
(see \cite{ACM}, Definition~3.3).

Unfortunately, these distributional traces are in general too weak
to be used in a chain rule (see the discussion in \cite{ACM} and \cite{ADM}).

For our purposes we need the notion of (strong) traces given below.

\begin{definition}[Traces]\label{def:tracce}
	Let $u\in L^\infty_{\rm loc}(\R^{N})$ and  
	let $\mathcal J\subset \R^{N}$ be a countably $\H^{N-1}$-rectifiable set oriented by a normal vector
	field $\nu$. 
	We say that two Borel functions \(u^{\pm}\colon\mathcal{J}\to\R\)
	are the traces of \(u\) on \(\mathcal{J}\)
	if for \(\H^{N-1}\)-almost every \(x\in\mathcal{J}\)
	it holds
	\[
	\lim_{r\to 0^+} \int_{B_r^\pm(x)} |u(y) - u^\pm(x)|\, dy = 0,
	\]
	where
	$B_r^\pm(x) :=  B_r(x) \cap \{ y\in\R^{N} : \pm \langle y-x,\nu(x)\rangle \geq 0 \}$.
\end{definition}

The following result is a particular case of Lemma~3.1 in \cite{CDDG}.

\begin{lemma}\label{lemma:campo2}
	Let $\mathcal J$ be a countably $\mathcal{H}^{N-1}$--rectifiable set 
	with oriented normal vector $\nu$.
	Let \(\A\in \mathcal{DM}^{\infty}\)
	and assume that $\A$ admits traces $\A^\pm$ on $\mathcal J$
	for \(\mathcal H^{N-1}\) almost every $z\in\mathcal J$.
	Then
	it holds
	\[
	\Div \A\res \mathcal J =  \pscal{\A^+ - \A^-}{\nu}\, \mathcal H^{N-1}\res \mathcal J\,.
	\]  
\end{lemma}

\medskip
\noindent

\subsection{Assumptions on  the vector field \(\bder\).}
\label{ss:assumptions}
Let $\bder:\R^{N}\times\bbbr\rightarrow\bbbr^N$ be a locally
bounded Borel function satisfying the following conditions:
\begin{enumerate}
\item[(i)]  for ${\mathcal L}^N$-a.e.\ $x\in
\R^{N}$ the function $\bder(x,\cdot)$ is continuous in $\bbbr$,
uniformly w.r.t.\ $x$;
\item[(ii)] for $\LLU$-a.e.\ $t\in
\bbbr$ the function $\bder(\cdot,t)  \in \DM$;
\item[(iii)] the measure 
\begin{equation}\label{defsigmabarra}\nonumber
\sigma\defeq \bigvee_{t\in \R} |\Div_x \bder(\cdot,t)|
\end{equation}
is a Radon measure;
\item[(iv)] for any compact set $H\subset\R$ there exist a positive Radon measure
$\lambda_H$ and a modulus of continuity $\omega_H$ such that
\[
|\widetilde\Div_x \bder(\cdot,t)(C)-\widetilde\Div_x \bder(\cdot,w)(C)|\le
\omega_H(|t-w|)\lambda_H(C)
\]
for all $t,w\in H$ and $C\subset\R^{N}$ Borel.
\end{enumerate}

Let us define the singular set 
\begin{equation}
\label{f:N}
{\mathcal N}=\Big\{x\in \R^{N}: \liminf_{r \to 0} \frac{\sigma(B_r(x))}{r^{N-1}}>0\Big\}.
\end{equation}
We assume that
\begin{itemize}
	\item[(v)]
	$\mathcal N$ is a countably $\mathcal{H}^{N-1}$--rectifiable set and
	$\H^{N-1}(\mathcal N\cap K)<+\infty$
	for every compact set $K\subseteq \R^{N}$.
\end{itemize}

In the following, \(\nu\) will always denote an oriented normal vector field on \(\mathcal{N}\).
We remark that, in the setting of Theorem~\ref{chain ruleqwqw},
the rectifiability of the singular set $\mathcal N$
follows from (i)--(iv).

Furthermore we assume that 
\begin{itemize}
\item[(vi)] 
for every $t\in\R$ and for \(\H^{N-1}\) almost every $x\in\R^{N}\setminus \mathcal N$  there exists the limit
\begin{equation}\label{f:vi}
\widetilde \bder(x,t)=\lim_{r\to 0} \mean{B_r(x)} \bder(y,t)\,dy.
\end{equation}
\end{itemize}

Without loss of generality we shall always assume that  
$\bder(x,t)=\widetilde \bder(x,t)$
on points \((x,t)\) where \eqref{f:vi} holds.

By using (vi), as in \cite[Section  3]{ACDD},  
we can prove that there exists a set $\mathcal N_0 $ with \(\H^{N-1}(\mathcal N_0)=0\) such that for every point \(x\in\R^{N}\setminus (\mathcal N\cup \mathcal N_0)\) and every $\w\in \R$  there exists the limit
\[
\widetilde \bder(x,t)=\lim_{r\to 0} \mean{B_r(x)} \bder(y,t)\,dy.
\]

Moreover we consider the following assumption on the traces of the vector field \(\bder\).

\begin{itemize}
\item[(vii)] 
For every \(t\in\R\), 
the function \(\bder(\cdot, t)\) admits (strong) traces
\(\bder^\pm(\cdot,t)\) on \(\mathcal{N}\).
\end{itemize}
In what follows, we shall use the notation 
\(\beta^\pm(x,t) := \pscal{\bder^\pm(x,t)}{\nu(x)}\),
\(t\in\R\), \(x\in\mathcal{N}\).

\begin{remark}\label{anzellotti}
By assumption (ii) it follows that, for every \(t\in\R\),
the vector field \(\bder(\cdot, t)\) admits distributional traces
\(\text{Tr}^\pm(\bder(\cdot, t), \mathcal{N})\) on \(\mathcal{N}\)
in the sense of Anzellotti
(see \cite{Anz,cf}).
Nevertheless, this notion of trace is too weak
in order to obtain the chain rule formula.	
\end{remark}

\medskip
As in \cite[Prop.~3.2(ii)]{ACDD},
it can be proved that there exists
a Borel set \(\mathcal{N}_1\subseteq \mathcal{N}\)
such that, for every \(x\in\mathcal{N}_1\),
the traces \(\bder^\pm(x,t)\) are defined for every \(t\in\R\) and are continuous in \(t\).

By (vii) and Proposition \ref{lemma:campo2} we have
\begin{equation}
\label{mmm}
\frac{d\Div_x^j\bder(\cdot,t)}{d\hh}(x)=\beta^+(x,t)-\beta^-(x,t)
\end{equation}
for every $\w\in \R$ and for $\H^{N-1}$ a.e. $x\in {\mathcal N}$.

If assumptions (i)--(vii) hold, then for every $t\in\R$ the decomposition formula 
\begin{equation}
\label{decomp}
(\Div_x \bder)(\cdot,\w)=(\Div^a_x \bder)(x,\w)\,{\mathcal L}^{N}+\frac{\Div_x^c\bder(\cdot,t)}{d\sigma}(x)\sigma+
\left[\beta^{+}(x,\w)-\beta^{-}(x,\w)\right]\, {\mathcal H}^{N-1} \res{{\mathcal N}}
\end{equation}
holds in the sense of measures and there exists a Borel set ${\mathcal N_2}\subseteq \R^{N}$ with $\sigma({\mathcal N_2})=0$ such that 
the following limit
\[
\lim_{r\downarrow  0}\frac{\widetilde\Div_x\bder(\cdot,t)(B_r(x))}{\sigma(B_r(x))}
=\frac{d\widetilde\Div_x\bder(\cdot,t)}{d\sigma}(x)
\]
exists for every $x\in \R^{N}\setminus {\mathcal N_2}$ and for every $t\in\R$  and this equality holds,
where $\frac{d\Div_x\bder(\cdot,t)}{d\sigma}(x)$ is the Radon-Nikod\'ym derivative at $x$ of the measure $\Div_x\bder(\cdot,t)$ w.r.t.\ $\sigma$.
In particular we have that there exists a Borel set ${\mathcal N_3}\subseteq \R^{N}$ with $\LLN({\mathcal N_3})=0$ such that 
the following limit
$$
\lim_{r\downarrow  0}\frac{\Div^a_x\bder(\cdot,t)(B_r(x))}{\LLN(B_r(x))}=\frac{d\Div^a_x\bder(\cdot,t)}{d\LLN}(x)
$$
exists for every $x\in \R^{N}\setminus {\mathcal N_3}$ and for every $t\in\R$  and this equality holds,
where $\frac{d\Div^a_x\bder(\cdot,t)}{d\sigma}(x)$ is the Radon-Nikod\'ym derivative at $x$ of the measure $\Div^a_x\bder(\cdot,t)$ w.r.t.\ $\LLN$. 
Similarly, there exists a Borel set ${\mathcal N_4}\subseteq \R^{N}$ with $\sigma({\mathcal N_4})=0$ such that 
the following limit
$$
\lim_{r\downarrow  0}\frac{\Div^c_x\bder(\cdot,t)(B_r(x))}{\sigma(B_r(x))}=\frac{d\Div^c_x\bder(\cdot,t)}{d\sigma}(x)
$$
exists for every $x\in \R^{N}\setminus {\mathcal N_4}$ and for every $t\in\R$  and this equality holds,
where $\frac{d\Div^c_x\bder(\cdot,t)}{d\sigma}(x)$ is the Radon-Nikod\'ym derivative at $x$ of the measure $\Div^c_x\bder(\cdot,t)$ w.r.t.\ $\sigma$.

\subsection{Preliminary results}
\label{ss:prel}

Let us define
\begin{equation}\label{defB}
\B(x,t)=\int_0^t \bder(x,w)\,dw\,.
\end{equation}

\begin{proposition}\label{kkkk}
 Let $\bder:\R^{N}\times\R\rightarrow\R^N$ be a Borel function satisfying {\rm (i)--(vii)}.
We have that
\begin{enumerate}
\item[(a)]
for every $t\in\R$ the function $x\mapsto \B(x,t)\,dt$ belongs to $\mathcal{DM}^{\infty}\,;$
\item[(b)] for every $t\in\R$ 
$$
\Div_x\B(x, t)<\!\!<\sigma\,;
$$ 
\item[(c)] the equality 
$$\Div_x\B(x, t)=
\left[\int_{0}^{t}
\frac{d\Div_x\bder}{d\sigma}(x,w)\,dw\right]\,d\sigma
$$
holds in the sense of measures for every $t\in\R$.
\end{enumerate} 
\end{proposition}

\begin{proof} Since
$$
\left|\frac{d \Div_x\bder}{d \sigma}(x,w)
\right|\leq 1\,,
$$
for every test function $\phi\in C^1_0(\R^{N})$, for 
 $\H^{N-1}$ a.e. $x\in\R^{N}$
and for every $t\in\R$, we have
\begin{equation}\nonumber%
\begin{split}
& \int_{\R^{N}}\nabla\phi(x)\B(x, t)\,dx=
\int_{0}^{t}\left[\int_{\R^{N}}\nabla\phi(x)
\bder(x,w)\,dx\right]dw
\\
=&-\int_{0}^{t}\left[\int_{\R^{N}}\phi(x)
d \Div_x\bder(\cdot,w)\right]dw
=-\int_{0}^{t}\left[\int_{\R^{N}}\phi(x)
\frac{d \Div_x\bder}{d \sigma}(x,w)\,d \sigma\right]dw
\\
=&-\int_{\R^{N}}\phi(x)\left[\int_{0}^{t}
\frac{d \Div_x\bder}{d \sigma}(x,w)\,dw\right]d \sigma.
\qedhere
\end{split}
\end{equation}
\end{proof}

\begin{corollary} \label{bbbb} Under the previous assumptions, we have
\[
\Div^a_x\B(x, t) =
\int_{0}^{t}
\frac{d\Div^a_x\bder}{d\LLN}(x,w)\,dw
\quad\forall x\in\R^N\setminus\mathcal{N}_3,
\]
and 
$$\Div^c_x\B(x, t)=
\left[\int_{0}^{t}
\frac{d\Div^c_x\bder}{d\sigma}(x,w)\,dw\right]\,d\sigma
$$
in the sense of measures.
\end{corollary}

\begin{corollary}\label{lololo}
Let $\bder$ be a Borel function satisfying {\rm (i)--(vii)}
and let \(\B\) be the vector field defined in \eqref{defB}.
Then it holds:
\begin{enumerate}
\item[(a)]
For every \(x\in{\mathcal N}\setminus {\mathcal N_1}\)
and for every \(w\in\R\)
one has
\[
\lim_{r\downarrow 0}
\mean{B^\pm_r(x)}
|\B(y,w) - \B^\pm(x,w)|\, dy = 0, 
\]
where
\[
\B^\pm(x,w) := \int_0^w \bder^{\pm}(x,t)\, dt.
\]

\item[(b)] 
For every $x\in\R^{N}\setminus({\mathcal N}\cup\mathcal N_1)$
and every \(w\in\R\) one has
\[
\lim_{r\downarrow 0} \mean{B_r(x)}\left|\B(y,w)-\widetilde{\B}(x,w)\right|\,dy=0,
\]
where
\[
\widetilde{\B}(x,w) := \int_0^w \bder(x,t)\, dt.
\]

\item[(c)] 
The equality 
\[
\begin{split}
\Div_x \B(x,t)
= {}&\left[\int_{0}^{w}
\nabla_x\bder(x,t)\,dt\right]
\,d\LLN
+\left[\int_{0}^{w}
\frac{dD_x^c\bder}{d\sigma}(x,t)\,dt\right]
\,d\sigma
\\
& +
\pscal{\B^+(x,t)-\B^-(x,t)}{\nu(x)}
\,d\hh\res{{{\mathcal N}}}
\end{split}
\]
holds in the sense of measures.
\end{enumerate} 
\end{corollary}

\begin{proof} 
(a) By (vii) for every $x\in {\mathcal N}\setminus\mathcal N_1$ and for every $w\in\R$ we have
\begin{equation}\nonumber%
\begin{split}
\lim_{r\downarrow 0} &
\mean{B^\pm_r(x)}\left|\int_{0}^{w}\bder(y,t)\,dt-\int_{0}^{w}\bder^\pm(x,t)\,dt\right|dy\\
\leq &\int_{0}^{w}\lim_{r\downarrow 0} \mean{B^\pm_r(x)}|\bder(y,t)-\bder^\pm(x,t)|\,dy\, \,dt=0.
\end{split}
\end{equation}

\noindent
(b) Similarly for every $x\in\R^{N}\setminus( {\mathcal N}\cup\mathcal N_1)$ and for every $w\in\R$
we have
\begin{equation}\nonumber%
\begin{split}
\lim_{r\downarrow 0} & \mean{B_r(x)}\left|\int_{0}^{w}\bder(y,t)\,dt-\int_{0}^{w}\widetilde \bder(x,t)\,dt\right|dy\\
\leq &\int_{0}^{w}\lim_{r\downarrow 0} \mean{B_r(x)}|\bder(y,t)-\widetilde \bder(x,t)|\,dy\, \,dt=0.
\end{split}
\end{equation}

\noindent
(c) Follows from (a), (b) and Corollary~\ref{bbbb}.
\end{proof}

\subsection{Main result}
\label{ss:div2}

\begin{theorem}[$\mathcal{DM}^{\infty}$-dependence]\label{chainb3}Let $\bder:\R^{N}\times\bbbr\rightarrow\bbbr^N$ be a locally
bounded Borel function satisfying conditions (i)--(vii). Then, for every $u\in BV_{\rm loc}(\R^{N})\cap L^{\infty}_{\rm loc}(\R^{N})$,
the function $\v:\R^{N}\rightarrow\bbbr^N$,
defined by
\[
\v(x):=\B(x,u(x))
\,,
\]
belongs to $\mathcal{DM}^{\infty}$ and for any
$\phi\in C_0^1(\R^{N})$ we have
\[
\begin{split}
\int_{\R^{N}} & \pscal{\nabla \phi(x)}{\v(x)}
 \, dx = {} \\
-&\int_{\R^{N}}\phi(x)\Div_x^a\B(x,u(x))\,dx
-\int_{\R^{N}}\phi(x)\frac{\Div_x^c\B}{d\sigma}(x,\widetilde u(x))\,d\sigma
\\ 
-&\int_{\R^{N}}\!\phi(x)\pscal{\bder(x,\widetilde u(x))}{\nabla u(x)}\,dx
-\int_{\R^{N}}\phi(x)\pscal{\bder(x, \ut(x))}{\frac{D^c u}{|Du|} (x)}\, d\, |Du|(x)
\\  -&
\int_{{\mathcal N\cup J_u}}\phi(x)\pscal{
	\B^+(x,u^+(x))
	-\B^-(x,u^-(x))
}{\nu(x)}
\,d\hh\,\,.
\end{split}
\]
\end{theorem}

\begin{proof}
\textsl{Step~1.} 
Let us fix a test function \(\phi\in C^1_0(\R^{N})\).
We claim that
\[
\begin{split}
\int_{\R^{N}} \pscal{\nabla \phi(x)}{\v(x)} \, dx = {} 
& -\int_{\R} dt \int_{\R^{N}} \sign(t) \chi^*_{\Omega_{u,t}} \phi(x) \, d\, \Div_x \bder(x,t)
\\ & 
- \int_{\R^{N}} \phi(x)\, \pscal{\bder(x, \ut(x))}{\nabla u(x)}\, dx
\\ & 
- \int_{\R^{N}} \phi(x)\, \pscal{\bder(x, \ut(x))}{\frac{D^c u}{|Du|} (x)}\, d\, |Du|(x)
\\  & -
\frac12\int_{\mathcal N \cup J_u} \phi(x)
\left[\int_{u^-(x)}^{u^+(x)} \left[\beta^+(x,t) + \beta^-(x,t)\right]dt\right]\, d\H^{N-1}\,,
\end{split}
\]
where 
$\Omega_{u,t}=\{x\in\R^{N}:\,t$ belongs to the segment of
endpoints $0$ and $u(x)\}$ and $\chi^*_{\Omega_{u,t}}$ is the precise representative of the $BV$ function
$\chi_{\Omega_{u,t}}$.

In order to prove the claim,
it is enough
to use a regularization argument as in \cite{DCFV2}.
More precisely, if \(\bder_\epsilon(\cdot, t) := \rho_\epsilon \ast \bder(\cdot, t)\)
denotes the standard regularization of \(\bder(\cdot, t)\),
and \(\v_\epsilon(x) := \int_{0}^{u(x)} \bder_\epsilon(x,t)\, dt\), then
\[
\begin{split}
\int_{\R^{N}} \pscal{\nabla\phi(x)}{\v_\epsilon(x)}\, dx = {} &
-\int_{\R^{N}} \phi(x) \, \int_0^{u(x)} \Div_x \bder_\epsilon(x,t)\, dt
\\ & -
\int_{\R^{N}} \phi(x)\, \pscal{\bder_\epsilon(x, u)}{\nabla u}\, dx
\\ & -
\int_{\R^{N}} \phi(x)\, \pscal{\bder_\epsilon(x, \ut(x))}{\frac{D^c u}{|Du|} (x)}\, d\, |Du|(x)
\\  & -
\int_{\mathcal N \cup J_u} \phi(x)\,
\int_{u^-(x)}^{u^+(x)} \pscal{\bder_\epsilon(x,t)}{\nu(x)}\, dt
\, d\H^{N-1}(x)\,,
\end{split}
\] 
and the claim follows by passing to the limit as \(\epsilon\to 0^+\), observing that,
by (vii),
for every \(t\in\R\) it holds
\[
\lim_{\epsilon\to 0^+} \pscal{\bder_\epsilon(x,t)}{\nu(x)} =
\frac{\beta^+(x,t)+\beta^-(x,t)}{2}
\qquad
\text{for}\ \H^{N-1}-a.e.\ x\in\mathcal{N}.
\]
(For details see the proofs of Theorems 1.1 and 3.4 in \cite{DCFV2}.)

\medskip\noindent
\textsl{Step~2.}
We assume for simplicity that $u\geq 0$.
We claim that
\begin{equation} 
\begin{split}\nonumber
&\int_{\R}\left[\int_{\R^{N}}{\rm sgn}(t)\phi(x)
\chi^*_{\Omega_{u,t}}(x)\,d\Div_x \bder((x,t)\right]
\,dt
\\ \nonumber
& =\int_{\R^{N}}\phi(x)\Div_x^a\B(x,u(x))\,dx
-\int_{\R^{N}}\phi(x)\frac{d\Div_x^c\B}{d\sigma}(x,\widetilde u(x))\,d\sigma
\\ \nonumber
&\quad +\frac12\int_{{\mathcal N\cup J_u}}\phi(x)\left[\int_{0}^{u^+(x)}[\beta^+(x,t)-\beta^-(x,t)]\,dt
+\int_{0}^{u^-(x)}[\beta^+(x,t)-\beta^-(x,t)]\,dt\right]
\,d\hh\,.
\nonumber
\end{split}
\end{equation}
By using Lemma 2.2 in \cite{DCFV2} and by the decomposition formula \eqref{decomp}, we have
\begin{equation} 
\begin{split}\nonumber
& \int_{-\infty}^{+\infty} \left[\int_{\R^{N}}{\rm sgn}(t)\phi(x)
\chi^*_{\Omega_{u,t}}\,d\Div_x\bder
(x,t)\right]
\,dt =
\int_{0}^{+\infty}\left[\int_{\R^{N}}\phi(x)\chi^*_{\{{u>t\}}}\,d\Div_x\bder(x,t)\right]\,dt 
\\ \nonumber
& =\int_{0}^{+\infty}\left[\int_{\R^{N}}\phi(x)
\chi_{\{u>t\}}\,\Div^a_x\bder(x,t)dx\right]
\,dt
+\int_{0}^{+\infty}\left[\int_{\R^{N}}\phi(x)
\chi_{\{\widetilde u>t\}}\,\frac{d\Div_x^c\bder}{d\sigma}(x,t)\,d \sigma\right]
\,dt
\\ \nonumber
&\quad +\int_{0}^{+\infty}\left[\int_{{\mathcal N}\cup J_u}\phi(x)
\frac12\left[\chi_{\{u^+>t\}}+\chi_{\{u^->t\}}\right]\,[\beta^+(x,t)-\beta^-(x,t)]d\hh\right]
\,dt
\,.
\nonumber
\end{split}
\end{equation}
By using Corollary \ref{bbbb}, the claim is proved.

\medskip
Finally,
the thesis follows from Step 1 and Step 2, observing that 
\[ 
\begin{split}
&\frac12\int_{u^-(x)}^{u^+(x)} \left[\beta^+(x,t) + \beta^-(x,t)\right]dt\\
&+ \frac12\left[\int_{0}^{u^+(x)}[\beta^+(x,t)-\beta^-(x,t)]\,dt
+\int_{0}^{u^-(x)}[\beta^+(x,t)-\beta^-(x,t)]\,dt\right]\\
& =\int_{0}^{u^+(x)}\beta^+(x,t)\,dt
-\int_{0}^{u^-(x)}\beta^-(x,t)\,dt
\,.
\qedhere
\end{split}
\]
\end{proof}

\smallskip

\begin{corollary}\label{c:div}
Let \(h\in C^1(\R)\) be a function with bounded derivative,
let \(\HH\in \mathcal{DM}^{\infty}\) and
let 
\[
{\mathcal N}=\Big\{x\in \R^{N}: \liminf_{r \to 0} \frac{|\Div \HH | (B_r(x))}{r^{N-1}}>0\Big\}.
\]
Assume that:
\begin{itemize}
\item[(a)] 
	$\mathcal N$ is a countably $\mathcal{H}^{N-1}$--rectifiable set and
	$\H^{N-1}(\mathcal N\cap K)<+\infty$
	for every compact set $K\subseteq \R^{N}$.

\item[(b)]
For every \(x\in\R^N\setminus\mathcal{N}\),
one has \(\HH(x) = \widetilde{\HH}(x)\).

\item[(c)]
The vector field \(\HH\) admits strong traces
\(\HH^\pm\) on \(\mathcal{N}\).
\end{itemize}

Then, for every $u\in BV(\R^{N})\cap L^{\infty}(\R^{N})$,
the function $\v:\R^{N}\rightarrow\bbbr^N$,
defined by
\[
\v(x):=\HH(x)h(u(x))
\,,
\]
belongs to $\mathcal{DM}^{\infty}$ and for any
$\phi\in C_0^1(\R^{N})$ we have
\[
\begin{split}
&\int_{\R^{N}} \pscal{\nabla \phi(x)}{\v(x)}
 \, dx = {} \\
-&\int_{\R^{N}}\phi(x)\Div^a\HH(x)h(u(x))\,dx
-\int_{\R^{N}}\phi(x)\frac{\Div^c\HH}{d\sigma}(x)h(\widetilde u(x))\,d\sigma
\\ 
-&\int_{\R^{N}}\!\phi(x)h^\prime(\widetilde u(x))\pscal{\HH(x)}{\nabla u(x)}\,dx
-\int_{\R^{N}}\phi(x)h^\prime(\widetilde u(x))\pscal{\HH(x)}{\frac{D^c u}{|Du|} (x)}\, d\, |Du|(x)
\\  -&
\int_{{\mathcal N\cup J_u}}\phi(x)
\left[h(u^+(x))\pscal{\HH^+(x)}{\nu(x)}-h(u^-(x))\pscal{\HH^-(x)}{\nu(x)}
\right]
\,d\hh\,\,.
\end{split}
\]
\end{corollary}

\begin{remark}
The theory of divergence--measure vector fields is due
to G.~Anzellotti \cite{Anz} 
(see also G.--Q.~Chen and
H.~Frid \cite{cf} for its generalization
and \cite{ScSc}). 
He introduced the ``dot product" of a bounded vector field $\A$, whose divergence is a Radon measure,  and the gradient $Du$ of $u\in BV(\Omega)$ through a pairing $(\A,Du)$ which defines a Radon measure. 
He also defined the normal trace of a vector field through the boundary and establish a generalized Gauss--Green formula.

Consider now $\mu=\hbox{\rm div}\,\A$ with $\A\in\DM$ and let $u\in
BV_{\rm loc}(\R^N)\cap L^\infty_{\rm{loc}}(\R^N)$.
The distribution defined by the following expression
\[
\langle(\A,Du),\varphi\rangle=-\int u^*\varphi\,d\mu-\int
 u\A\cdot\nabla\varphi,\quad \varphi\in C_0^\infty(\R^N)\,,
\]
is actually a Radon measure and its total variation $\vert (\A, Du) \vert$ is absolutely
continuous with respect to the measure $\vert Du \vert$.
Therefore the following Anzellotti formula holds
\begin{equation}\label{Anzellotti}
\Div(u\A)=u^*\Div\A+(\A,Du).
\end{equation}
in the sense of measures.
Using the distributional normal trace
defined in \eqref{f:ntrace},
the following Green formula holds
\begin{equation}\label{GreenI}
\int_{\Omega} u^* \, d\mu + \int_{\Omega} (\A, Du) =
\int_{\partial \Omega} 
\text{Tr}(\A, \partial\Omega) \,
u \ d\mathcal H^{N-1}\,.
\end{equation}
We remark that,
if we apply \eqref{GreenI} to a vector field $\A \in \mathcal{DM}^{\infty}(\Omega)$ and the constant $u\equiv 1$, since $(\A, Du)=0$ we obtain
\begin{equation}\label{GreenII}
\int_{\Omega}  \Div\A =
\int_{\partial \Omega} 
\text{Tr}(\A, \partial\Omega)
\ d\mathcal H^{N-1}\,.
\end{equation}
\end{remark}

\begin{remark}\label{remrem}
The case \(u\in W^{1,1}\) in Theorem~\ref{chainb3} has been already treated
in \cite{DCFV2} (see also \cite{dcl}).
The representation formula in Theorem~\ref{chainb3} can be written as the following equality in the sense of measures
\begin{equation}\label{lolo}
\begin{split}
\Div(\B(x,u(x)))=&\left.\Div_x\B(x,t)\right|_{t=u(x)}\\
+&\pscal{\bder(x, \ut(x))}{\widetilde Du}
+\pscal{
	\B^*(x,u^+)
	-\B^*(x,u^-)
}{\nu}
\hh\res {\mathcal N\cup J_u}
\end{split}
\end{equation}
with the compact notation
$$
\left.\Div_x \B(\cdot,t)\right|_{t=u(x)}=\frac12\left[\Div_x \B(\cdot,u^+(x))+\Div_x \B(\cdot,u^-(x))\right]\,.
$$
\end{remark}

\begin{remark}
The formula in Corollary~\ref{c:div} can be written as the following equality in the sense of measures:
\[
\Div(\HH(x)h(u(x)))=\Div\HH(x)\,h(u)^*(x)+h^\prime(u(x))\pscal{\HH(x)}{Du}.
\]
If $u\in BV$ and $u\HH\in \mathcal{DM}^{\infty}$, then $$\pscal{\HH(x)}{Du}=-u^* \Div\HH+\Div(u\HH).$$
Hence the following formula holds
\[
\begin{split}
\Div(\HH(x)h(u))=&\Div\HH\,h(u)^*+h^\prime(u)[-u \Div\HH+\Div(u\HH)]
\\
=&[h(u)^*-uh^\prime(u)]\Div\HH+h^\prime(u)\Div(u\HH).
\end{split}
\]
(For a similar formula when $u$ is not a function of bounded variation
see \cite{ADM}.)

If $\B(x,u(x))=\HH(x)u(x)$, then we obtain the Anzellotti formula \eqref{Anzellotti}.
\end{remark}

\section{Applications to conservation laws}
\label{s:cons}

Let us consider the multidimensional scalar conservation law with discontinuous flux
\begin{equation}\label{f:scalar}
u_t + \Div \B(x,u) = 0, 
\qquad (t,x)\in (0, +\infty)\times\R^N,
\end{equation}
where, as in the previous section,
the vector field \(\B\) is defined as in \eqref{defB} by
\(\B(x,t) := \int_0^t \bder(x,w)\, dw\).
Here the Borel function \(\bder\colon\R^N\times\R\to\R^N\)
satisfies slightly stronger assumptions compared with
(i)--(vii) of Section~\ref{s:div}.

We are interested in proving the Kato contraction property for
entropy solutions, in the sense of Definition~\ref{d:entrsol}, of \eqref{f:scalar},
see Theorem~\ref{thm:uniquenessentropic}.
We stress here that, as in \cite{CDD}, our definition of entropy solution is restricted to
BV functions.

For the sake of completeness we collect the assumptions on \(\bder\) here.
(A prime denotes the fact that the assumption has been modified
with respect to the corresponding one listed in
Section~\ref{ss:assumptions}.)

\begin{enumerate}
\item[(i)]  for ${\mathcal L}^N$-a.e.\ $x\in\R^{N}$ the function $\bder(x,\cdot)$ is continuous in $\bbbr$
uniformly w.r.t.\ $x$;
\item[(ii')] for $\LLU$-a.e.\ $t\in
\bbbr$ the function $\bder(\cdot,t)$ belongs to $\mathcal{SDM}^{\infty}$;
\item[(iii)] the measure 
$\sigma$ defined in \eqref{defsigmabarra}
is a Radon measure.
\item[(iv')] 
there exists a function \(g_1\in L^1_{\rm loc}(\R^N)\) such that
\[
|\Div^a_x \bder(x,t)-\Div^a_x \bder(x,w)|\leq
g_1(x) \, |t-w|
\]
for all $t,w\in \R$ and $x\in\R^{N}$;

\item[(v')] 
the set $\mathcal N$, defined in \eqref{f:N}, is a countably $\mathcal{H}^{N-1}$--rectifiable set and
$\H^{N-1}(\mathcal N)<+\infty$;

\item[(vi)] for every $t\in\R$ and for \(\H^{N-1}\) almost every $x\in\R^{N}\setminus \mathcal N$  there exists the limit~\eqref{f:vi};

\item[(vii)] for every \(t\in\R\), 
the function \(\bder(\cdot, t)\) admits (strong) traces
\(\bder^\pm(\cdot,t)\) on \(\mathcal{N}\).
\end{enumerate}
As in Section~\ref{s:div}, we shall use the notation 
\(\beta^\pm(x,t) := \pscal{\bder^\pm(x,t)}{\nu(x)}\),
\(t\in\R\), \(x\in\mathcal{N}\).

In our context, 
Theorem~\ref{chainb3} reads as follows: 
For every \(u\in BV(\R^{N})\) the composite function \(\v(x)=\B(x ,u(x))\) belongs to  \(\mathcal{DM}^{\infty}\) with
\begin{equation}\label{eq:stimab}
|\Div \v|\le \sigma+M|Du|
\end{equation}
and 
\begin{align}
\label{eq:chainrule1b}
\widetilde{\Div}\, \v & =
\Div_x^a \B (x,\widetilde u(x))\, \Leb{N}+
\pscal{\bder (x,\widetilde u(x))}{\widetilde{D} u}
\\
\label{eq:chainrule2b}
\Div^j \v &= 
\pscal{\B^+(x,u^+(x))-\B^{-}(x,u^-(x))}{\nu(x)}\,\mathcal H^{N-1}\res (J_u\cup \mathcal N).
\end{align}

\begin{remark}\label{lincei}
Let us  point out that our hypotheses include (and actually are modeled on) the case \(\B(x,u)=\widehat\B (k(x),u)\) where \(k\in S\mathcal{DM}^{\infty}\cap L^\infty(\R^{N}; \R^{N}) \), \(\H^{N-1}(J_{k})<+\infty\) and 
\(\widehat \B\in C^1(\R^{N}\times \R,\R^{N})\cap{\rm Lip} (\R^{N}\times \R,\R^{N}) \).
\end{remark}

\begin{remark}\label{r:triang}
	The assumption (vii) of existence of traces 
	may appear
	too strong to be useful for applications.	
	On the other hand, a situation we have in mind is the following.
	Let us consider the system
	\begin{equation}\label{f:triang}
	\begin{cases}
	\Div \A_1(u) = 0,\\
	\Div \A_2(u,v) = 0,
	\end{cases}
	\end{equation}
	where \(u,v\colon \R^{N}\to \R\),
	and the fluxes
	\(\A_1\colon\R\to\R^{N}\), \(\A_2\colon \R\times\R\to\R^{N}\)
	are regular functions.
	In general, a solution \(u\) of the first equation need not be of bounded variation.
	Nevertheless, if \(\A_1\) is genuinely nonlinear, then
	\(u\) has a quasi-BV structure, in the sense of
	De Lellis, Otto and Westdickenberg (see \cite{DLOW}).
	In particular, there exists a $\mathcal{H}^{N-1}$--rectifiable set \(\mathcal{N}\)
	such that
	\(u\) has left and right traces \(u^\pm\) on \(\mathcal{N}\),
	and it has vanishing mean oscillation at every 
	\(x\not\in\mathcal{N}\).
	The second component \(v\) is then a solution of the equation
	\[
	\Div \B(x, v) = 0
	\]
	where \(\B(x, v) := \A_2(u(x), v)\).
	The vector field \(\B\) admits traces on \(\mathcal{N}\), given by
	\(\B^{\pm}(x, v) = \B(u^\pm(x), v)\), \(x\in \mathcal{N}\), \(v\in\R\).
	In particular, assumption (vii) is satisfied.
	Unfortunately, the quasi-BV structure of \(u\)
	is not enough to have Assumption~\ref{ss:assumptions} (vii) satisfied,
	since \(\B(\cdot, v)\) is only of bounded mean oscillation at every
	point \(x\not\in \mathcal{N}\).
	Nevertheless, we think that our analysis can be a good starting point
	in order to obtain uniqueness results for the Cauchy problems
	related to the evolutionary version of the triangular system \eqref{f:triang}.
\end{remark}


\subsection{Uniqueness of entropy solutions}

\begin{definition}[Convex entropy pair]
We say that $(S,\eeta)$ is a convex entropy pair
if $S\in C^2(\R)$ is a convex function,
and $\eeta = (\eta_1, \ldots, \eta_N)$ is defined by
\begin{equation}\label{f:eta}
\eta_{i}(x,v):=\int_{0}^{v} b_{i}(x,w)S'(w)dw\,,
\qquad i=1,\ldots,N.
\end{equation}
\end{definition}
In the above definition and in the sequel, 
\(b_i=\bder\cdot e_i\) are the components of \(\bder\).

Note that according to the previous discussion, \(\eeta(\cdot, v) \in S\mathcal{DM}^{\infty}\) for every \(v\in \R\) and its divergence is given by
\[
\begin{split}
\Div_x\eeta(\cdot,v)&=
\left(\int_0^v \Div^a_x\bder(x,w)\,S'(w)dw\right) \Leb{N}
\\ &
+\left(\int_0^v 
(\beta^+(x,w) - \beta^-(x,w))\, S'(w)\, dw
\right) \,\mathcal H^{N-1}\res \mathcal N\,.
\end{split}
\]

\begin{definition}[Entropy solutions]\label{d:entrsol}
A function 
\[
u\in C([0,T]; L^1(\R^{N})) \cap L^{\infty}((0,T)\times \R^{N})\cap 
BV((0,T)\times\R^N)
\]
is an entropy solution of \eqref{f:scalar}
if 
$u$ is a solution to \eqref{f:scalar} in the sense of distributions,
and there exists a  (everywhere defined) Borel representative
$\uh$ of $u$ with \(|\uh(t,x)|\le \|u\|_{\infty}\) such that, for every 
convex entropy pair $(S,\eeta)$,
one has
\begin{equation}\label{f:distr}
\begin{split}
\partial_t S(u)&+\Div \big(\eeta(x,u)\big)\\
&-\Div \eeta(x,v)\Big|_{v=\uh(t,x)}+S'(\uh)\Div \B(x,v)\Big|_{v=\uh(t,x)}\le 0
\end{split}
\end{equation}
in  the distributional sense. Here, by  \(\Div \B(x,v)\big|_{v=\uh(t,x)}\)
we mean the measure whose action on a bounded and Borel function \(\varphi=\varphi(t,x)\) is given by 
\begin{multline}\label{def:misuraresto}
\int _0^T dt \int_{\R^{N}} \varphi(t,x) 
\Div^a_x \B(x,\uh(t,x)) \, dx
\\
 +\int _0^T dt \int_{\mathcal  N} \varphi(t,x)
 \pscal{\B^+(x,\uh(t,x))-\B^-(x,\uh(t,x))}{\nu(x)}
 \, d\H^{N-1}(x),
\end{multline}
and the same for \( \Div \eeta(x,v)\big|_{v=\uh(t,x)}\).
\end{definition}

With these definitions at hand we can now restate our main result:

\begin{theorem}\label{thm:uniquenessentropic}   
Let \(\bder\) satisfy
the assumptions listed at the beginning of the Section,
and let  \(u_1\) and \(u_2\) be two entropy solutions of \eqref{f:scalar}, then
\begin{equation}\label{eq:contrattivo}
\int_{\R^{N}}|u_1(T,x)-u_2(T,x)|\,dx\le \int_{\R^{N}}|u_1(0,x)-u_2(0,x)|\,dx.
\end{equation}
\end{theorem}

\subsection{Sketch of the proof of Theorem~\ref{thm:uniquenessentropic}}

The proof follows the lines of the one of Theorem~3.5 in \cite{CDD}.
We recall here only the main points,
giving references to the details in \cite{CDD}.
As a first reduction, it is not restrictive to consider
only non-negative solutions.

\medskip\noindent
\textbf{Kinetic formulation.}
Let us consider the measure-valued vector field
\[
\boldsymbol{a}(\cdot, v) := 
(\bder(\cdot, v)\, \LLN_x \, , \, -\Div_x \B(\cdot, v)).
\]
Note that \(\boldsymbol{a}\) is a Radon measure and
\(\Div_{x,v} \boldsymbol{a} = 0\).

By a kinetic solution we mean a function
\[
u\in C([0,T]; L^1(\R^{N})) \cap L^{\infty}((0,T)\times \R^{N})\cap 
BV((0,T)\times\R^N)
\]
which is a distributional solution of \eqref{f:scalar},
and satisfies the following property:
there exists a  (everywhere defined) Borel representative
$\uh$ of $u$ with \(|\uh(t,x)|\le \|u\|_{\infty}\),
and a positive measure \(m = m(t,x,v)\) with
\(m((0,T)\times\R^{N+1}) < +\infty\), such that the function
\((t,x,v) \mapsto \chi(v, \uh(t,x))\) satisfies
\begin{equation}
\label{f:kinetic}
\partial_t  \chi( v, \uh(t,x)) +
\Div_{x,v} [ \boldsymbol{a}(x,v)  \chi(v, \uh(t,x))]
= \partial_v m(t,x,v)
\end{equation}
in the sense of distributions.
Here the function \(\chi\) is defined by
\begin{equation}\label{defchi}
\chi(v,u):=
\begin{cases}
1
&\text{if $v<u$},\\
1/2
&\text{if $v=u$},\\
0
&\text{if $u<v$}.
\end{cases}
\end{equation}

The first step consists in proving that $u$
is an entropy solution of \eqref{f:scalar} if and only if it is
a kinetic solution.

The proof is almost the same of the one of Theorem~3.9 in \cite{CDD}.
We mention only that, if $u$ is an entropy solution, the kinetic measure $m$
is defined first of all as a distribution by
\begin{equation}\label{eq:8}
\begin{split}
\pscal{m}{\psi} & =
- \int_{(0,T)\times\R^{N+1}} dt\, dx\, dv \, \partial_t \psi(t,x,v)\int_0^v \chi(w,u(t,x))\,  dw \\
& -\int_{(0,T)\times\R^{N+1}} dt\, dx\, dv \, \nabla_x \psi(t,x,v)  \int_0^v \bder(x,w) \chi(w,u(t,x)) \, dw 
\\
& + \int_{(0,T)\times\R^{N+1}} dt\, dx\, dv 
\, \B(x,v) \chi(v, \uh(t,x)) \nabla_x \psi(t,x,v).
\end{split}
\end{equation}
Reasoning as in \cite{CDD} we can prove that $m$ is a positive measure, with finite mass.

\medskip\noindent
\textbf{Kato contraction property.}
By Cavalieri's principle,
it is enough to prove the following contraction
property
(see \cite[Thm.~5.1]{CDD}).

\begin{theorem}\label{kin:uniqueness}
Let $u_1, u_2$ be two 
entropy solution of \eqref{f:scalar},
with corresponding everywhere defined Borel 
representatives $\uh_1, \uh_2$.
Setting $f_i(t,x,v) := \chi(v,\uh_i(t,x))$, $i=1,2$, we have that
\[
\int_{\R^{N+1}}\ 
|f_1-f_2|(T,x,v) 
\ dx\,dv 
\leq \int_{\R^{N+1}}\ 
|f_1-f_2|(0,x,v)
\ dx\,dv \,.
\]
\end{theorem}

The proof of this theorem is obtained as a consequence of
the following intermediate result
(see \cite[Prop.~5.3]{CDD}).

\begin{proposition}\label{uniqueness1}
Let $u_1, u_2$ be two 
entropy solution of \eqref{f:scalar},
with corresponding representatives $\uh_1, \uh_2$.
Setting $f_i(t,x,v) := \chi(v,\uh_i(t,x))$, $i=1,2$, we have that
\begin{equation}\label{f:contr}
\begin{split}
\int_{\R^{N+1}}\ 
|f_1-f_2|(T,x,v) 
\ dx\,dv 
&\leq \int_{\R^{N+1}}\ 
|f_1-f_2|(0,x,v)
\ dx\,dv \\
&+\int_0^T \int_{\mathcal N} W(u_1,u_2)d\H^{N-1}dt
\,,
\end{split}
\end{equation}
where 
\begin{equation}\label{57bis}
\begin{split}
 W(u_1,u_2)
  := {} 
& \pscal{\B^+(u_1^+)}{\nu}\big[-2\chi(u_1^+,u_2^+)+2\chi(u_1^-,u_2^-)\big]\\
+ &
\pscal{\B^+(u_2^+)}{\nu}\big[-2\chi(u_2^+,u_1^+)+2\chi(u_2^-,u_1^-)\big]\,.
\end{split}
\end{equation}
\end{proposition}

The proof of this proposition is long and
technical, but can be done exactly as in \cite{CDD}
using the new chain rule proved
in Section~\ref{s:div}.

Finally, Theorem~\ref{kin:uniqueness} follows
from Proposition~\ref{uniqueness1}
by proving that $W(u_1, u_2) \leq 0$
on $(0,T)\times\mathcal{N}$.
This is a purely algebraic fact that follows from the
entropy condition and the existence of traces,
see the proof of Theorem~5.1 in \cite{CDD} for details.

\bigskip
\textsc{Acknowledgments.}
The authors would like to thank 
Guido De Philippis for some useful discussions
during the preparation of the manuscript.


\end{document}